\documentclass[a4paper,11pt,reqno,twoside]{amsart}

\usepackage{amsmath}
\usepackage{amsfonts}
\usepackage{amssymb}
\usepackage{amsthm}
\usepackage{color}
\usepackage{ifpdf}
\usepackage{array}
\usepackage{url}
\usepackage{multirow,bigdelim}
\usepackage{ascmac}

\numberwithin{equation}{section}

\newtheorem{theorem}{Theorem}[section]

\newtheorem{proposition}{Proposition}[section]
\newtheorem{corollary}{Corollary}[section]

\newcommand*{\C}{\mathbb{C}}%.............................C
\newcommand*{\R}{\mathbb{R}}%.............................R
\newcommand*{\Z}{\mathbb{Z}}%.............................Z
\newcommand*{\N}{\mathbb{N}}%.............................N
\newcommand{\comment}[1]{}
\title[The screw line of the Riemann zeta-function and its applications]%
      {The screw line of the Riemann zeta-function \\ and its applications} 
\author[M. Suzuki]{Masatoshi Suzuki}
\date{Version of \today}
\AtBeginDocument{%
\begin{abstract}
We investigate the screw line corresponding to the screw function 
associated with the Riemann zeta-function under the Riemann hypothesis 
and derive three necessary and sufficient conditions for the Riemann hypothesis as applications. 
One of them explains the non-negativity of the Weil distribution by means of the norm.
\end{abstract}
\maketitle
}
\begin{document}
%..........................................

%%%%%%%%%%%%%%%%%%%%%%%%%%%%%%%%%%%%%%%%%%%%%%%%%%%%%%%%%%%%%%%%
%
\section{Introduction} 
%
%%%%%%%%%%%%%%%%%%%%%%%%%%%%%%%%%%%%%%%%%%%%%%%%%%%%%%%%%%%%%%%%

Let $g(t)$ be the even real-valued function on the real line defined by 
\begin{equation} \label{eq_101}
\aligned 
g(t)
& := -4(e^{t/2}+e^{-t/2}-2) 
- \frac{t}{2}\left[ \frac{\Gamma'}{\Gamma}\left(\frac{1}{4}\right) - \log \pi \right] \\
&\quad  - 
\frac{1}{4}\left( \Phi(1,2,1/4) - e^{-t/2}\Phi(e^{-2t},2,1/4) \right)
 + \sum_{n \leq e^t} \frac{\Lambda(n)}{\sqrt{n}}(t-\log n)
\endaligned 
\end{equation}
for non-negative $t$, 
where $\Lambda(n)$ is the von Mangoldt function defined by 
$\Lambda(n)=\log p$ if $n=p^k$ with $k \in \Z_{>0}$ 
and $\Lambda(n)=0$ otherwise, 
$\Gamma(s)$ is the gamma function, 
and $\Phi(z,s,a) = \sum_{n=0}^{\infty} (n+a)^{-s}z^n$ 
is the Hurwitz--Lerch zeta-function. 

We assume that the Riemann hypothesis (RH, for short) 
%for the Riemann zeta-function $\zeta(s)$ 
is true, that is, all nontrivial zeros of the Riemann zeta-function $\zeta(s)$ 
lie on the critical line $\Re(s)=1/2$. 
Then, $-g(t)$ is non-negative for all $t \in \R$ and vice versa (\cite[Theorem 1.7]{Su22}). 
This non-negativity can be understood by the Weil distribution 
(\cite[Section 3.4]{Su22}). 
In this paper, we demonstrate that the non-negativity 
can be explained by the norm of a Hilbert space, 
building on the same theoretical background used in \cite{Su23} 
which explained the non-negativity of Li coefficients in terms of norms.
%On the other hand, the non-negativity is explained by using the norm of a Hilbert space as follows.
 
As stated in \cite[Theorem 1.2]{Su22}, 
the even function $g(t)$ is a {\it screw function} on the real line in the sense of \cite{KrLa14} 
if the RH is true, 
that is, the kernel defined by 
\begin{equation} \label{eq_0514_1}
G_g(t,u):=g(t-u)-g(t)-g(u)+g(0)
\end{equation}
is non-negative definite on the real line. 
If $g(t)$ is a screw function, 
then there exists a Hilbert space $\mathcal{H}$ 
and a continuous mapping $t \mapsto x(t)$ from $\R$ into $\mathcal{H}$ 
such that 
$\langle x(t+v)-x(v), x(u+v)-x(v) \rangle_{\mathcal{H}}$ 
is independent of $v \in \R$ for all $t,u \in \R$ 
and the equality 
$\langle x(t)-x(0), x(u)-x(0) \rangle_{\mathcal{H}} =G_g(t,u)$ holds.  
Therefore, $\Vert x(t)-x(0) \Vert_{\mathcal{H}}^2=-2g(t)$ 
if we note $g(0)=0$. 
A mapping $x:\R \to \mathcal{H}$ endowed with 
the translation-invariance described above is called a {\it screw line}. 
%Such a translation invariant mapping $x:\R \to \mathcal{H}$ is called a {\it screw line}.

One of the screw lines corresponding to $g(t)$ can be constructed as follows. 
In general, an even real-valued function $\tilde{g}(t)$ on $\R$ 
with $\tilde{g}(0)=0$ is a screw function if and only if 
it admits a representation 
\[
\tilde{g}(t) = \int_{-\infty}^{\infty} \frac{\cos(\gamma t)-1}{\gamma^2} \, d\tilde{\tau}(\gamma)
\]
with a non-negative measure $\tilde{\tau}$ such that 
$\int_{-\infty}^{\infty}d\tilde{\tau}(\gamma)/(1+\gamma^2)<\infty$. 
Hence, there exists a non-negative measure $\tau$ 
representing $g(t)$ as above under the RH. 
Then the Hilbert space $\mathcal{H}=L^2(\tau)$ and 
the mapping $t \mapsto x(t):=(e^{it\gamma}-1)/\gamma$ 
provide a screw line satisfying $\Vert x(t)-x(0) \Vert_{\mathcal{H}}^2=-2g(t)$ 
(\cite[\S12]{KrLa14}). 
This %kind of 
spectral construction for a screw line is important and useful in analysis, 
but it is not very useful for studying the nontrivial zeros of $\zeta(s)$ without assuming the RH. 
%Therefore, we attempt to construct a screw line 
%that does not depend on the measure $\tau$, at least superficially.
Therefore, we provide a construction of a screw line 
that is at least superficially independent of the measure $\tau$, 
and apply it to the RH.
\medskip

Let $L^2(\R)$ be the usual $L^2$-space on the real line with respect to the Lebesgue measure. 
Let $\xi(s)=2^{-1}s(s-1)\pi^{-s/2}\Gamma(s/2)\zeta(s)$ be the Riemann xi-function. 
The nontrivial zeros of  $\zeta(s)$ coincide with the zeros of $\xi(s)$ with multiplicity.
We define 
\begin{equation} \label{eq_102}
E(z):=\xi(1/2-iz)+\xi'(1/2-iz), 
\end{equation}
\begin{equation} \label{eq_103}
\Theta(z):=\overline{E(\bar{z})}/E(z), 
\end{equation}
and 
\begin{equation} \label{eq_104}
\mathfrak{S}_t(z) := \frac{i(1+\Theta(z))}{2\sqrt{\pi}}\, \mathfrak{P}_t(z)
\end{equation}
with
\begin{equation} \label{eq_105}
\aligned 
\mathfrak{P}_t(z)
& := \frac{4(e^{t/2}-1)}{1-2iz} + \frac{4(e^{-t/2}-1)}{1+2iz} 
 -  \sum_{n \leq e^{t}} \frac{\Lambda(n)}{\sqrt{n}} \frac{e^{iz(t-\log n)}-1}{iz} \\
& \quad +
\frac{e^{itz}-1}{iz} \left[ 
\frac{Z'}{Z}\left(\frac{1}{2}-iz \right)
-\frac{1}{2}\log\pi 
+\frac{1}{2}
\frac{\Gamma'}{\Gamma}\left(\frac{1}{4}+\frac{iz}{2}\right)
\right] \\
& \quad 
- \frac{1}{2iz}\left[
\frac{\Gamma'}{\Gamma}\left(\frac{1}{4}\right)
-
\frac{\Gamma'}{\Gamma}\left(\frac{1}{4}+\frac{iz}{2}\right)
\right] \\
& \quad -\frac{1}{2iz}e^{-t/2}
\Bigl[ \Phi(e^{-2t},1,1/4) -\Phi(e^{-2t},1,\tfrac{1}{2}(\tfrac{1}{2}+iz)) \Bigr] 
\endaligned 
\end{equation}
for a non-negative real number $t$ and a complex number $z$, 
where $Z(s):=\pi^{-s/2}\Gamma(s/2)\zeta(s)$. 
For negative $t$, we define $\mathfrak{S}_t(z):=\mathfrak{S}_{-t}(z)$. 
For this $\mathfrak{S}_t$, we first obtain the following. 

\begin{proposition} \label{prop_1}
For any fixed $t \in \R$, 
$\mathfrak{S}_t(z)$ belongs to $L^2(\R)$ as a function of $z$.
\end{proposition}

%This will be proved in Section \ref{section_2}. 
%
From this result, the mapping $t \mapsto \mathfrak{S}_t(z)$ from $\R$ to $L^2(\R)$ is defined.
For this mapping, the following holds under the RH.

\begin{theorem} \label{thm_1}
Assuming that the RH is true, the mapping 
$t \mapsto \mathfrak{S_t}(z)$ from $\R$ to $L^2(\R)$ is a screw line corresponding to $g(t)$. 
\end{theorem}

%This will be proved in Section \ref{section_3}. 
The following immediately follows from Theorem \ref{thm_1}. 

\begin{corollary} \label{cor_1}
The RH is true if and only if the equality 
\begin{equation} \label{eq_107}
\Vert \mathfrak{S}_t \Vert_{L^2(\R)}^2 = -2g(t)
\end{equation}
holds for all $t \geq t_0$ for some $t_0  \geq 0$. 
\end{corollary}

In \cite[Theorem 1.7]{Su22}, it has been proven that the non-negativity of $-g(t)$ 
is equivalent to the RH being true. 
Corollary \ref{cor_1} explains this non-negativity through a set of equalities involving norms. 
On the other hand, a screw function defines a non-negative definite hermitian form. 
As a variant of Corollary \ref{cor_1}, 
the non-negativity of the hermitian form associated with $g(t)$ 
is also explained through a set of equalities involving norms. 

The kernel \eqref{eq_0514_1} defines a hermitian form on 
the space $C_c^\infty(\R)$ of all smooth and compactly supported function on $\R$  by 
\begin{equation} \label{eq_0515_2}
\langle \phi_1, \phi_2 \rangle_{G_g} 
:= \int_{-\infty}^{\infty}\int_{-\infty}^{\infty}  G_g(t,u) \phi_1(u) \overline{\phi_2(t)}\, dtdu
\quad \text{for} \quad \phi_1,\,\phi_2 \in C_c^\infty(\R). 
\end{equation}
This hermitian form is non-negative definite if $g(t)$ is a screw function (\cite[\S5]{KrLa14}), that is, if the RH is true. 
By the uniformity of the $L^2$-norm of $\mathfrak{S}_t(z)$ 
on a compact set of $t$ obtained in the proof of Proposition \ref{prop_1} 
and Minkowski's integral inequality, we obtain the following.

\begin{proposition} \label{prop_102} 
For $\phi \in C_c^\infty(\R)$, 
we define 
\[
\widehat{\mathcal{P}_\phi}(z):=
\int_{-\infty}^{\infty} \mathfrak{S}_t(z)\phi(t)  \, dt
\]
using \eqref{eq_104}.  
Then $\widehat{\mathcal{P}_\phi}(z)$ belongs to $L^2(\R)$. 
\end{proposition}

Based on this proposition, the following holds. 

\begin{theorem} \label{thm_2} 
The RH is true if and only if the equality 
\begin{equation} \label{eq_0514_2}
\Vert \widehat{\mathcal{P}_\phi} \Vert_{L^2(\R)}^2 = \langle \phi, \phi \rangle_{G_g}
\end{equation}
holds for all $\phi \in C_c^\infty(\R)$ satisfying $\int_{-\infty}^{\infty} \phi(t)\, dt=0$. 
If the RH is true, equality \eqref{eq_0514_2} holds for all $\phi \in C_c^\infty(\R)$. 
\end{theorem}

The advantage of Corollary \ref{cor_1} and Theorem \ref{thm_2} 
is that it has turned the criterion of the RH 
from a set of inequalities like Weil's criterion into a set of equalities. 
Theorem \ref{thm_2} can be reformulated as follows.

We denote the set of all zeros of $\xi(1/2-iz)$ 
without counting multiplicity by $\Gamma$  
and the multiplicity of $\gamma \in \Gamma$ by $m_\gamma$. 
Weil's hermitian form on $C_c^\infty(\R)$ is defined by 
\[
\langle \psi_1, \psi_2 \rangle_W 
= \sum_{\gamma \in \Gamma} m_\gamma
\int_{-\infty}^{\infty} \psi_1(t)e^{i\gamma t}\,dt 
\overline{\int_{-\infty}^{\infty} \psi_2(u) e^{i \bar{\gamma} u} \, du }
\]
for $\psi_1, \psi_2 \in C_c^\infty(\R)$. 
The relation 
%between the hermitian forms $\langle \cdot, \cdot \rangle_{G_g}$ 
%and  $\langle \cdot, \cdot \rangle_{W}$ 
%
\begin{equation} \label{eq_0515_4}
\langle D\psi_1, D\psi_2 \rangle_{G_g} = \langle \psi_1, \psi_2 \rangle_W, 
\quad (D\psi)(t):=\psi'(t)
\end{equation}
of the hermitian forms 
in \cite[Proposition 3.1]{Su22} 
immediately leads to the following equivalence condition 
for the RH via Weil's criterion, 
since the differential operator $D$ gives a bijection from $C_c^\infty(\R)$ 
to the subspace of $C_c^\infty(\R)$ consisting of $\phi$ 
with $\int_{-\infty}^{\infty} \phi(t) \, dt=0$. 

\begin{corollary} \label{cor_2}
The RH is true if and only if the equality 
\begin{equation} \label{eq_0515_1}
\Vert \widehat{\mathcal{P}_{D\psi}} \Vert_{L^2(\R)}^2 = \langle \psi, \psi \rangle_{W}
\end{equation}
holds for all $\psi \in C_c^\infty(\R)$. 
\end{corollary}

A similar result was obtained in \cite[Theorem 1.3]{Su23a}, 
but Corollary \ref{cor_2} is simpler and more efficient as a statement. 
\medskip

In the following, 
we first prove one proposition (Proposition \ref{prop_201}) 
used in the proof of Proposition \ref{prop_1} and Theorem \ref{thm_1} 
in Section \ref{section_2}. 
Then, we prove Proposition \ref{prop_1} in the same section. 
After that, we prove Theorems \ref{thm_1} and \ref{thm_2} 
after preparing a result (Proposition \ref{prop_202}) on the theory of model spaces 
in Section \ref{section_3}. 
Finally, we mention two special values of $\mathfrak{S}_t(z)$ in Section \ref{section_4}. 
The strategy of the proof of Theorem \ref{thm_1} is similar to \cite{Su23}, 
however the computational details change. 
In \cite{Su23}, the analytic or geometric meaning of the functions 
that give the norms is unknown, but in this paper the functions that give the norms 
have the meaning as a screw line. 
Furthermore, as an advantage of using the screw line $\mathfrak{S}_t$, 
we obtain Theorem \ref{thm_2}, whose analogue was not obtained in \cite{Su23}. 

\section{ Unconditional results for $\mathfrak{S}_t$ }  \label{section_2}

\subsection{Expansion of $\mathfrak{P}_t(z)$ over the zeros} 
For the basic properties of the Riemann zeta-function, we refer to \cite{Tit86}. 
Let $\Gamma$ be the set of all zeros of $\xi(1/2-iz)$ without counting multiplicity. 
By two functional equations $\xi(s)=\xi(1-s)$ and $\xi(s)=\xi^\sharp(s)$, 
if $\gamma \in \Gamma$, then both $-\gamma$ and $\overline{\gamma}$ 
belong to $\Gamma$ with the same multiplicity.  
Also, $|\Im(\gamma)| < 1/2$ for every $\gamma \in \Gamma$, 
since all zeros of $\xi(s)$ lie in the strip $0 < \Re(s) < 1$. 
The RH is equivalent to all $\gamma \in \Gamma$ are real. 
For $E(z)$ of \eqref{eq_103}, we define 
\begin{equation}  \label{eq_201}
A(z) = (E(z)+\overline{E(\bar{z})})/2. 
\end{equation}
Then $A(z)=\xi(1/2-iz)$, because $\overline{E(\bar{z})}=\xi(1/2-iz)-\xi'(1/2-iz)$ 
by functional equations of $\xi(s)$. 
Therefore, $\Gamma$ coincides with the set of all zeros of both $A(z)$ and $1+\Theta(z)$. 
We define 
\begin{equation} \label{eq_202}
P_t(z) 
:= \sum_{\gamma \in \Gamma} m_\gamma \, \frac{e^{i\gamma t}-1}{\gamma}
\cdot
\frac{1}{z-\gamma}
\end{equation}
for $t \in \R_{\geq 0}$, 
where $m_\gamma$ is the multiplicity of $\gamma \in \Gamma$. 
For negative $t$, we set $P_{t}(z):=P_{-t}(z)$. 
The series on the right-hand side \eqref{eq_202} converges absolutely and uniformly 
on every compact subset of $\C\setminus\Gamma$, 
since $\sum_{\gamma \in \Gamma}m_\gamma|\gamma|^{-1-\delta}<\infty$ for any $\delta>0$, 
because $A(z)$ is an entire function of order one. 
%and $m_\gamma \ll \log \gamma$ by \cite[\S9.2, p.211]{Tit86}. 
Therefore, $P_t(z)$ is a meromorphic function on $\C$ 
with $\Gamma$ as the set of all poles.

\begin{proposition} \label{prop_201} 
Let $\mathfrak{P}_t(z)$ and $P_t(z)$ be meromorphic functions 
defined by \eqref{eq_105} and \eqref{eq_202}, respectively. 
Then, both coincide.
\end{proposition}
\begin{proof} For $t \geq 0$ and $z \in \C$ with $\Im(z)>0$, we define 
\[
\phi_{z,t}(x) = (iz)^{-1} \,e^{izx} (e^{izt}-e^{-iz\min(0,x)})\,\mathbf{1}_{(-t,\infty]}(x), 
\]
where $\mathbf{1}_A(x)=1$ if $x \in A$ and $\mathbf{1}_A(x)=0$ otherwise. 
The main tool for the proof is Weil's explicit formula
\begin{equation*} 
\aligned 
\lim_{X \to \infty} & \sum_{{\gamma \in \Gamma}\atop{|\gamma|\leq X}} m_\gamma
\int_{-\infty}^{\infty} \phi(x) \, e^{-i\gamma x} \, dx \\
& = \int_{-\infty}^{\infty} \phi(x) (e^{x/2}+e^{-x/2}) dx 
 - \sum_{n=1}^{\infty} \frac{\Lambda(n)}{\sqrt{n}} \phi(\log n) 
- \sum_{n=1}^{\infty} \frac{\Lambda(n)}{\sqrt{n}} \phi(-\log n)  \\
& \quad 
- (\log 4\pi + \gamma_0) \phi(0)
- \int_{0}^{\infty} \left\{ \phi(x) + \phi(-x)-2e^{-x/2} \phi(0)\right\} \frac{e^{x/2}dx}{e^{x}-e^{-x}}
\endaligned 
\end{equation*}
which is obtained from the explicit formula in \cite[p. 186]{Bo01} 
by taking $\phi(t) = e^{t/2}f(e^t)$ for test functions $f(x)$ in that formula
with the conditions for $f(x)$ in \cite[Section 3]{BoLa99}, 
where $\gamma_0$ is the Euler--Mascheroni constant. 
(Note that the formula in \cite{BoLa99} has two typographical errors 
in the second line of the right-hand side.) 

As is easily seen, Weil's explicit formula can be applied to $\phi(x)=\phi_{z,t}(x)$. 
We have  
\[
\int_{-\infty}^{\infty} \phi_{z,t}(x) \, e^{-i\gamma x} \, dx 
= \frac{e^{i\gamma t}-1}{\gamma}
\cdot
\frac{1}{z-\gamma} 
\quad \text{when $\Im(z)>\Im(\gamma)$}. 
\]
Therefore, the left-hand side of Weil's explicit formula for $\phi_{z,t}(x)$ 
gives $P_t(z)$ of \eqref{eq_202} when $\Im(z)>1/2$. 
Hence, if it is shown that the right-hand side 
is equal to $\mathfrak{P}_t(z)$ for $\Im(z)>1/2$, 
then the conclusion of the proposition follows by analytic continuation. 

It is easy to verify
\[
\int_{-\infty}^{\infty} \phi_{z,t}(x) (e^{x/2}+e^{-x/2}) dx 
= \frac{4(e^{t/2}-1)}{1-2iz} + \frac{4(e^{-t/2}-1)}{1+2iz}
\]
and 
\[
\aligned 
\sum_{n=1}^{\infty} \frac{\Lambda(n)}{\sqrt{n}} \, & \phi_{z,t}(\log n) 
=\frac{e^{izt}-1}{iz} \sum_{n=1}^{\infty} \frac{\Lambda(n)}{n^{1/2-iz}}
= - \frac{e^{izt}-1}{iz} \frac{\zeta'}{\zeta}\left( \frac{1}{2}-iz \right), \\
\sum_{n=1}^{\infty} \frac{\Lambda(n)}{\sqrt{n}} \, & \phi_{z,t}(-\log n) 
= \sum_{n \leq e^{t}} \frac{\Lambda(n)}{\sqrt{n}} \frac{e^{iz(t-\log n)}- 1}{iz}
\endaligned 
\]
for $\Im(z)>1/2$ by direct calculation. 

Therefore, the remaining task is to calculate the fifth term on the right-hand side. 
We split it into $\int_{t}^{\infty}$ and $\int_{0}^{t}$ and calculate each integral. 
For the first part, 
\[
\aligned 
\int_{t}^{\infty} & \left\{ \phi_{z,t}(x) + \phi_{z,t}(-x)-2e^{-x/2} \phi_{z,t}(0)\right\} 
\frac{e^{x/2}dx}{e^{x}-e^{-x}} \\
& = \frac{e^{izt}-1}{iz}  
\int_{t}^{\infty} (e^{izx} -2e^{-x/2})\frac{e^{x/2}dx}{e^{x}-e^{-x}} \\
& = \frac{e^{izt}-1}{iz}  
\int_{t}^{\infty} e^{izx} \,e^{-x/2} \sum_{n=0}^{\infty}e^{-2nx} \, dx
- \frac{e^{izt}-1}{iz}  \log \frac{e^t+1}{e^t-1} \\
& = \frac{e^{izt}-1}{iz} 
\left[ \frac{1}{2}\,e^{-t(\frac{1}{2}-iz)}
\sum_{n=0}^{\infty} \frac{e^{-2nt}}{n+\frac{1}{2}(\frac{1}{2}-iz)}
-  \log\coth(e^{t/2}) \right] \\
& = \frac{e^{izt}-1}{iz} 
\left[ \frac{1}{2}\,e^{-t(\frac{1}{2}-iz)}
\Phi(e^{-2t},1,\tfrac{1}{2}(\tfrac{1}{2}-iz))
-  \log\coth(e^{t/2})\right] . 
\endaligned 
\]
For the second part, 
\[
\aligned 
\int_{0}^{t} & \left\{ \phi_{z,t}(x) + \phi_{z,t}(-x)-2e^{-x/2} \phi_{z,t}(0)\right\} \frac{e^{x/2}dx}{e^{x}-e^{-x}} \\
& = \frac{1}{iz}\int_{0}^{t} \left\{ 
(e^{iz(t-x)}-1) + (e^{izx}-2e^{-x/2}) (e^{izt}-1)
\right\} \,e^{-x/2} \sum_{n=0}^{\infty} e^{-2nx}\, dx.
\endaligned 
\]
To handle the first half of this right-hand side, we calculate as 
\[
\aligned 
\int_{0}^{t} 
(e^{iz(t-x)}-1) & \,e^{-x/2} \sum_{n=0}^{N} e^{-2nx}\, dx \\
& = \frac{1}{2}
\sum_{n=0}^{N} \frac{e^{itz}-e^{-t/2}e^{-2nt}}{n+\frac{1}{2}(\frac{1}{2}+iz)}
- \frac{1}{2}\sum_{n=0}^{N} \frac{1-e^{-\frac{1}{2}(1+4n)t}}{n+\frac{1}{4}} \\
& = \frac{1}{2}e^{itz}
\sum_{n=0}^{N} \frac{1}{n+\frac{1}{2}(\frac{1}{2}+iz)}
-\frac{1}{2}e^{-t/2}\Phi(e^{-2t},1,\tfrac{1}{2}(\tfrac{1}{2}+iz)) \\
& \quad - \frac{1}{2}\sum_{n=0}^{N} \frac{1}{n+\frac{1}{4}} 
+ \frac{1}{2}e^{-t/2}\Phi(e^{-2t},1,1/4)+O(e^{-2Nt}), 
\endaligned 
\]
where the implied constant depends on $t$ and $z$. 
To handle the second half of the right-hand side, we calculate as 
\[
\aligned 
\int_{0}^{t} & (e^{izx}-2e^{-x/2})\,e^{-x/2} \sum_{n=0}^{N} e^{-2nx}\, dx \\
&= 
\frac{1}{2} \sum_{n=0}^{N} \frac{1-e^{-t(\frac{1}{2}-iz)}e^{-2nt}}{n+\frac{1}{2}(\frac{1}{2}-iz)}
- 
\sum_{n=0}^{N} \frac{1-e^{-t}e^{-2nt}}{n+\frac{1}{2}} \\
& = 
\frac{1}{2} \sum_{n=0}^{N} \frac{1}{n+\frac{1}{2}(\frac{1}{2}-iz)}
-
\frac{1}{2} e^{-t(\frac{1}{2}-iz)} \Phi(e^{-2t},1,\tfrac{1}{2}(\tfrac{1}{2}-iz))\\
& \quad - 
\sum_{n=0}^{N} \frac{1}{n+\frac{1}{2}}
+ \log\coth(e^{t/2})+O(e^{-2Nt}), 
\endaligned 
\]
where we used the series expansion of ${\rm arctanh}(e^{-t})=2^{-1}\log\coth(e^{t/2})$ 
and  the implied constant depends on $t$ and $z$. 

By the above preliminary calculations and the well-known series expansion
\begin{equation} \label{eq_203} 
\frac{\Gamma'}{\Gamma}(w) = -\gamma_0 - \sum_{n=0}^{\infty}
\left( \frac{1}{w+n} - \frac{1}{n+1} \right), 
\end{equation}
we obtain
\[
\aligned 
iz\int_{0}^{t} & \left\{ \phi_{z,t}(x) + \phi_{z,t}(-x)-2e^{-x/2} \phi_{z,t}(0)\right\} 
\frac{e^{x/2}dx}{e^{x}-e^{-x}} \\
& = 
\frac{1}{2}e^{-t/2}\Phi(e^{-2t},1,1/4) 
-\frac{1}{2}e^{-t/2}\Phi(e^{-2t},1,\tfrac{1}{2}(\tfrac{1}{2}+iz)) \\
& \quad -
(e^{izt} -1)
\frac{1}{2} e^{-t(\frac{1}{2}-iz)} \Phi(e^{-2t},1,\tfrac{1}{2}(\tfrac{1}{2}-iz)) 
+  (e^{izt} -1)\log\coth(e^{t/2}) \\
& \quad +
(e^{itz}-1)\lim_{N \to \infty}\left[ 
\frac{1}{2}
\sum_{n=0}^{} \frac{1}{n+\frac{1}{2}(\frac{1}{2}+iz)} 
+ 
\frac{1}{2} \sum_{n=0}^{N} \frac{1}{n+\frac{1}{2}(\frac{1}{2}-iz)}
 -
\sum_{n=0}^{N} \frac{1}{n+\frac{1}{2}} 
\right] \\
& \quad 
+ \frac{1}{2} \lim_{N\to\infty} 
\left[\,  
\sum_{n=0}^{N} \frac{1}{n+\frac{1}{2}(\frac{1}{2}+iz)} 
- \sum_{n=0}^{N} \frac{1}{n+\frac{1}{4}} 
\,\right] \\
& = 
\frac{1}{2}e^{-t/2}\Phi(e^{-2t},1,1/4) 
-\frac{1}{2}e^{-t/2}\Phi(e^{-2t},1,\tfrac{1}{2}(\tfrac{1}{2}+iz)) \\
& \quad -
(e^{izt} -1)
\frac{1}{2} e^{-t(\frac{1}{2}-iz)} \Phi(e^{-2t},1,\tfrac{1}{2}(\tfrac{1}{2}-iz)) 
+ (e^{izt} -1)\log\coth(e^{t/2})\\
& \quad +
(e^{itz}-1) \frac{1}{2}\left[
2\frac{\Gamma'}{\Gamma}\left(\frac{1}{2}\right)
-
\frac{\Gamma'}{\Gamma}\left(\frac{1}{4}+\frac{iz}{2}\right)
-
\frac{\Gamma'}{\Gamma}\left(\frac{1}{4}-\frac{iz}{2}\right)
\right] \\
& \quad 
+ \frac{1}{2}\left[
\frac{\Gamma'}{\Gamma}\left(\frac{1}{4}\right)
-
\frac{\Gamma'}{\Gamma}\left(\frac{1}{4}+\frac{iz}{2}\right)
\right].
\endaligned 
\]
Combining the results for $\int_{t}^{\infty}$ and $\int_{0}^{t}$, 
\[
\aligned 
\int_{0}^{\infty} & \left\{ \phi_{z,t}(x) + \phi_{z,t}(-x)-2e^{-x/2} \phi_{z,t}(0) 
\right\} \frac{e^{x/2}dx}{e^{x}-e^{-x}} \\
& = \frac{1}{2iz}e^{-t/2}
\Bigl[ \Phi(e^{-2t},1,1/4) -\Phi(e^{-2t},1,\tfrac{1}{2}(\tfrac{1}{2}+iz)) \Bigr] \\
& \quad +
\frac{e^{itz}-1}{iz} \left[
\frac{\Gamma'}{\Gamma}\left(\frac{1}{2}\right)
-\frac{1}{2}
\frac{\Gamma'}{\Gamma}\left(\frac{1}{4}+\frac{iz}{2}\right)
-\frac{1}{2}
\frac{\Gamma'}{\Gamma}\left(\frac{1}{4}-\frac{iz}{2}\right)
\right] \\
& \quad 
+ \frac{1}{2iz}\left[
\frac{\Gamma'}{\Gamma}\left(\frac{1}{4}\right)
-
\frac{\Gamma'}{\Gamma}\left(\frac{1}{4}+\frac{iz}{2}\right)
\right].
\endaligned 
\]
Finally, noting the special value  $(\Gamma'/\Gamma)(1/2)=-\gamma_0 - 2\log 2$, 
we conclude that the right-hand side of Weil's explicit formula 
for $\phi_{z,t}(x)$ is equal to \eqref{eq_105}.
\end{proof}

\subsection{Proof of Proposition \ref{prop_1}} 

%First, we prove that $\mathfrak{S}_t(z)$ belongs to $L^2(\R)$ unconditionally 
%by using \eqref{eq_104} and \eqref{eq_105}. 
%
We have $|\Theta(z)|=1$ for every $z \in \R$ by definition. 
In fact, zeros of $E(z)$ in the denominator cancel out in the numerator $\overline{E(\bar{z})}$, 
even if they exist. 
Further, $\mathfrak{P}_t(z)$ has poles of order one at $\gamma \in \Gamma$, 
but $\mathfrak{S}_t(z)$ is holomorphic there, since 
$(1+\Theta(z))/2 = A(z)/E(z) = A(z)/(A(z)+iA'(z)) = (z-\gamma)(-i/m_\gamma+o(1)) 
$ near $z=\gamma$ by direct calculation.  
Hence, $\mathfrak{S}_t(z)$ is bounded and holomorphic on the real line 
by \eqref{eq_104}, \eqref{eq_202}, and Proposition \ref{prop_201}. 
On the other hand, in the horizontal strip $|\Im(z)|\leq 1/2$, 
we have the well-known estimate $(\Gamma'/\Gamma)(1/4+iz/2) \ll \log |z|$ 
and 
\[
\frac{\zeta'}{\zeta}\left(\frac{1}{2}-iz \right)
= \sum_{|\Re(z)-\gamma| \leq 1} \frac{i}{z-\gamma}+O(\log|z|)
\]
by \cite[Theorem 9.6 (A)]{Tit86}. 
In both estimates, implied constants are uniform in $|\Im(z)|\leq 1/2$. 
The number of zeros $\gamma \in \Gamma$ satisfying $|\Re(z)-\gamma| \leq 1$ 
is $O(\log |z|)$ counting with multiplicity by \cite[Theorem 9.2]{Tit86}. 
Therefore, 
$\mathfrak{S}_t(z) \ll |z|^{-1}\log |z|$ as $|z| \to \infty$ 
with an implied constant depending on a compact set of $t$ by \eqref{eq_105}. 
Hence $\mathfrak{S}_t(z)$ belongs to $L^2(\R)$ 
and the norm is uniformly bounded on a compact set of $t$. 
\hfill $\Box$

\section{Proof of the main results} \label{section_3}

\subsection{Preparation on the theory of the model spaces} 
For this part, we refer to \cite[Section 3.1]{Su23}, including precise definitions of notions.
Let $\mathbb{H}^2$ be the Hardy space on the upper half-plane. 
As usual, 
we identify $\mathbb{H}^2$ with a closed subspace of $L^2(\R)$ via boundary values. 
Then, the inner product of $\mathbb{H}^2$ coincides with the standard inner product of $L^2(\R)$. 

Assuming the RH is true, 
$E(z)$ of \eqref{eq_102} is an entire function 
satisfying $|E(\bar{z})|<|E(z)|$ if $\Im(z)>0$ (\cite[Theorem 1]{La06}). 
Therefore, it generates the de Branges space $\mathcal{H}(E)$, 
which is a Hilbert space of entire functions isomorphic 
to the model subspace 
$\mathcal{K}(\Theta):=\mathbb{H}^2 \ominus \Theta \mathbb{H}^2$ %of $L^2(\R)$ 
by the mapping $F(z) \mapsto F(z)/E(z)$ from $\mathcal{H}(E)$ into $\mathbb{H}^2$, 
where $\Theta(z)$ is the meromorphic function defined in \eqref{eq_103}. 
The model subspace $\mathcal{K}(\Theta)$ 
is a subspace of $L^2(\R)$ as a Hilbert space. 
In particular, the inner product of $\mathcal{K}(\Theta)$ matches that of $L^2(\R)$.

\begin{proposition} \label{prop_202}
Assuming the RH is true, the family 
\begin{equation} \label{eq_204}
F_\gamma(z) :=
\sqrt{\frac{m_\gamma}{\pi}} \frac{i(1+\Theta(z))}{2(z-\gamma)}, \quad \gamma \in \Gamma
\end{equation}
forms an orthonormal basis of the Hilbert space $\mathcal{K}(\Theta)$. 
\end{proposition}
\begin{proof} 
See \cite[Proposition 3.1]{Su23}. 
\end{proof}

\subsection{Proof of Theorem \ref{thm_1}} 

By Proposition \ref{prop_201}, we have 
\begin{equation} \label{eq_205} 
\mathfrak{S}_t(z)
= 
\sum_{\gamma \in \Gamma} \sqrt{m_\gamma} \, \frac{e^{i\gamma t}-1}{\gamma} \, F_\gamma(z) 
\end{equation}
unconditionally. Further, the coefficients on the right-hand side 
converge in $L^2$-sense:
\[
\sum_{\gamma \in \Gamma}  \left|\sqrt{m_\gamma} 
\frac{e^{i\gamma t}-1}{\gamma}
 \right|^2 
\leq \sum_{\gamma \in \Gamma} 
\frac{m_\gamma}{|\gamma|^2} < \infty. 
\]

Therefore, 
assuming the RH and applying Proposition \ref{prop_202} 
to $\mathfrak{S}_t(z)$ via formula \eqref{eq_205}, 
we find that 
$\mathfrak{S}_t(z)$ belongs to the subspace $\mathcal{K}(\Theta)$ of $L^2(\R)$ and 
\begin{equation} \label{eq_0512}
\langle \mathfrak{S}_{t+u}-\mathfrak{S}_{u}, \mathfrak{S}_{s+u}-\mathfrak{S}_{u} \rangle_{L^2(\R)}
= 
\sum_{\gamma \in \Gamma} m_\gamma \, 
\frac{e^{i\gamma t}-1}{\gamma}\cdot\frac{e^{-i\gamma s}-1}{\gamma}
\end{equation}
holds. The right-hand side is equal to $G_g(t,s)$ by \cite[(1.9)]{Su22}.  
%if $m_\gamma=1$ for all $\gamma \in \Gamma$. 
Hence, $\mathfrak{S}_t:\R \to L^2(\R)$ 
is a screw line of $g(t)$ under the RH. 

We find that $\mathfrak{S}_0(z)$ is identically zero by \eqref{eq_104} and \eqref{eq_105}, 
since 
\[
\lim_{t  \to 0} (\Phi(e^{-2t},1,1/4) -\Phi(e^{-2t},1,(1/2+iz)/2)) = 
-
\frac{\Gamma'}{\Gamma}\left(1/4\right)
+
\frac{\Gamma'}{\Gamma}\left((1/2+iz)/2\right) 
\]
by \eqref{eq_203}. Therefore, by taking $u=0$ in \eqref{eq_0512}, 
\begin{equation} \label{eq_301}
\aligned 
\Vert \mathfrak{S}_t \Vert_{L^2(\R)}^2 
& = \sum_{\gamma \in \Gamma} m_\gamma 
\left| \frac{e^{i\gamma t}-1}{\gamma} \right|^2 
=  2 \sum_{\gamma \in \Gamma} m_\gamma \, \frac{1-\cos(\gamma t)}{\gamma^2}. 
\endaligned 
\end{equation}
On the other hand, 
\begin{equation} \label{eq_302} 
-g(t)= \sum_{\gamma \in \Gamma} m_\gamma\, \frac{1-\cos(\gamma t)}{\gamma^2}
\end{equation}
by \cite[Theorem 1.1 (2)]{Su22}. 
Hence equality \eqref{eq_107} follows from \eqref{eq_301} and \eqref{eq_302}. 
\hfill $\Box$

\subsection{Proof of Corollary \ref{cor_1}}  

Theorem \ref{thm_1} states that \eqref{eq_107} is a necessary condition for the RH. 
Therefore, it suffices to show that \eqref{eq_107} is a sufficient condition for the RH.

We suppose that equality \eqref{eq_107} holds for all $t \geq t_0$.  
Then $-g(t)$ is nonnegative on $[t_0,\infty)$, which implies 
that  the RH is true by \cite[Theorems 1.7 and 11.1]{Su22}. 
\hfill $\Box$

\subsection{Proof of Theorem \ref{thm_2}} 

First, we show equation \eqref{eq_0514_2} 
assuming that the RH is true.  
For any $\phi \in C_c^\infty(\R)$, we have 
\begin{equation*} 
\widehat{\mathcal{P}_\phi}(z)
= 
\sum_{\gamma \in \Gamma} \sqrt{m_\gamma} \, 
\frac{\widehat{\phi}(\gamma)-\widehat{\phi}(0)}{\gamma} \, F_\gamma(z) 
\end{equation*}
by \eqref{eq_205}, where $\widehat{\phi}(z) := \int_{-\infty}^{+\infty} \phi(t) \,e^{izt} \, dt$. 
Therefore, 
\begin{equation} \label{eq_0515_3} 
\Vert \widehat{\mathcal{P}_\phi} \Vert_{L^2(\R)}^2 
= \sum_{\gamma \in \Gamma} m_\gamma
\left\vert \frac{\widehat{\phi}(\gamma)-\widehat{\phi}(0)}{\gamma} \right\vert^2
\end{equation}
by Proposition \ref{prop_202}. Applying 
\begin{equation*}
G_g(t,u) 
 = \sum_{\gamma} \frac{(e^{i\gamma t}-1)(e^{-i\gamma u}-1)}{\gamma^2}
\end{equation*}
in \cite[(1.9)]{Su22} to \eqref{eq_0515_2} 
and noting the symmetry $\gamma \mapsto -\gamma$, 
we find that the right-hand side of \eqref{eq_0515_3}  is equal to 
$\langle \phi, \phi \rangle_{G_g}$. 

Conversely, we show that the RH is true assuming equality \eqref{eq_0514_2}. 
We show that a contradiction arises if the RH is false.  
We take a non-real $\gamma_0 \in \Gamma$. 
For any $\epsilon>0$, there exists $\psi_1$, $\psi_2\in C_c^\infty(\R)$ 
such that $\widehat{\psi_1}(\gamma_0)=i$, 
$\widehat{\psi_2}(\overline{\gamma_0})=-i$, 
$|\widehat{\psi_1}(\gamma)| \leq \epsilon |\gamma_0-\gamma|^{-1-\delta}$ 
for every $\gamma \in \Gamma\setminus\{\gamma_0\}$, 
and  
$|\widehat{\psi_2}(\gamma)| \leq \epsilon |\overline{\gamma_0}-\gamma|^{-1-\delta}$ 
for every $\gamma \in \Gamma\setminus\{\overline{\gamma_0}\}$ by \cite[Lemma 1]{Yo92}. 
We define $\psi:=\psi_1+\psi_2 \,(\not=0)$ and $\phi:=D\psi$.  
Then, $\widehat{\phi}(0)=0$ by definition and 
$
\langle \phi, \phi \rangle_{G_g} 
=
\langle \psi, \psi \rangle_W 
$
by \eqref{eq_0515_4}. The right-hand side is equal to 
$
\sum_{\gamma \in \Gamma} m_\gamma 
\widehat{\psi}(\gamma)\overline{\widehat{\psi}(\bar{\gamma})}
= -m_{\gamma_0} + O(\epsilon)
$,  
since $\sum_{\gamma \in \Gamma} m_\gamma |\gamma|^{-1-\delta}<\infty$. 
Therefore, $\langle \phi, \phi \rangle_{G_g}$ is negative for a sufficiently small $\epsilon>0$, 
but it contradicts the non-negativity that follows from \eqref{eq_0514_2}. 
Hence the RH is true. 
\hfill $\Box$

\section{Special values of the screw line $\mathfrak{S}_t(z)$} \label{section_4} 

The screw line $\mathfrak{S}_t(z)$ 
has the following unconditional relations with the screw function $g(t)$. 
It is interesting that they are not a special case of equations 
obtained from the general theory of screw functions.

\begin{theorem} 
Let $g(t)$ and $\mathfrak{P}_t(z)$ 
be functions of \eqref{eq_101} and \eqref{eq_105}, respectively.  
Then the following equations hold independently of the truth of the RH: 
\begin{equation} \label{eq_401}
\mathfrak{P}_t(0) = -g(t), \\
\end{equation}
\begin{equation} \label{eq_402}
\lim_{y\to + \infty} \left[ y\, \mathfrak{B}_t(iy) -\frac{1}{2}
\frac{\Gamma'}{\Gamma}\left(\frac{1}{4}+\frac{y}{2} \right) 
+\frac{1}{2} \log \pi
\right] = -g'(t),  
\end{equation}
where we assume $t \not = \log n$ for any $n \in \N$ in \eqref{eq_402}.  
\end{theorem}

\begin{proof} 
Equality \eqref{eq_401} follows 
from \eqref{eq_202}, Proposition \ref{prop_201}, and \eqref{eq_302}, 
but it also follows directly from \eqref{eq_101} and \eqref{eq_105}. 
In fact, by $\Phi(z,s,a)=\sum_{n=0}^{\infty}z^n(n+a)^{-s}$ 
and \eqref{eq_203}, 
\[
\lim_{z\to 0}
\frac{1}{iz}
\Bigl[ \Phi(e^{-2t},1,1/4) -\Phi(e^{-2t},1,\tfrac{1}{2}(\tfrac{1}{2}+iz)) \Bigr] 
= \frac{1}{2} \Phi(e^{-2t},2,1/4), 
\]
\[
\lim_{z\to 0}
\frac{1}{iz}\left[
\frac{\Gamma'}{\Gamma}\left(\frac{1}{4}\right)
-
\frac{\Gamma'}{\Gamma}\left(\frac{1}{4}+\frac{iz}{2}\right)
\right]
= -\frac{1}{2}\psi_1\left(\frac{1}{4}\right), 
\]
where $\psi_1(z)$ is the polygamma function of order one.  
The expansion 
$\psi_1(w)=\sum_{n=0}^{\infty}(w+n)^{-2}$ 
gives $\psi_1(1/4)=\Phi(1,2,1/4)$. 
By $Z(s)=Z(1-s)$, we have $(Z'/Z)(1/2)=0$. 
Hence, by taking the limit $z \to 0$ in \eqref{eq_105}, 
we obtain the minus of \eqref{eq_101}.

To show \eqref{eq_402}, we multiply \eqref{eq_105} 
by $y$ and substitute $iy$ for $z$: 
\[
\aligned 
y\, \mathfrak{P}_t(iy)
& = \frac{4y(e^{t/2}-1)}{1+2y} + \frac{4y(e^{-t/2}-1)}{1-2y} 
 +  \sum_{n \leq e^{t}} \frac{\Lambda(n)}{\sqrt{n}} (e^{-y(t-\log n)}-1) \\
& \quad - (e^{-yt}-1) \left[ 
\frac{Z'}{Z}\left(\frac{1}{2}+y \right)
-\frac{1}{2}\log\pi 
+\frac{1}{2}
\frac{\Gamma'}{\Gamma}\left(\frac{1}{4}-\frac{y}{2}\right)
\right] \\
& \quad 
+ \frac{1}{2}\left[
\frac{\Gamma'}{\Gamma}\left(\frac{1}{4}\right)
-
\frac{\Gamma'}{\Gamma}\left(\frac{1}{4}-\frac{y}{2}\right)
\right] \\
& \quad +\frac{1}{2}e^{-t/2}
\Bigl[ \Phi(e^{-2t},1,1/4) -\Phi(e^{-2t},1,\tfrac{1}{2}(\tfrac{1}{2}-y)) \Bigr]. 
\endaligned 
\]
Therefore, for positive $t>0$, 
\[
\aligned 
\,&\lim_{y\to + \infty}  \left[ y\, \mathfrak{B}_t(iy) -\frac{1}{2}
\frac{\Gamma'}{\Gamma}\left(\frac{1}{4}+\frac{y}{2} \right) 
+\frac{1}{2} \log \pi
\right] \\
& = 2(e^{t/2}-e^{-t/2})  -  \sum_{n \leq e^{t}} \frac{\Lambda(n)}{\sqrt{n}} 
+ \frac{1}{2}\left[
\frac{\Gamma'}{\Gamma}\left(\frac{1}{4}\right)-\log\pi \right]+\frac{1}{2}e^{-t/2}\Phi(e^{-2t},1,1/4).
\endaligned 
\]
The right-hand side equals to $-g'(t)$ if $t\not=\log n$ 
by \eqref{eq_101} and 
$(d/dt)(e^{-t/2}\Phi(e^{-2t},2,1/4))=-2e^{-t/2}\Phi(e^{-2t},2,1/4)$ 
follows form $\Phi(z,s,a)=\sum_{n=0}^{\infty}z^n(n+a)^{-s}$.  
\end{proof}

\medskip

\noindent
{\bf Acknowledgments}~
This work was supported by JSPS KAKENHI Grant Number JP17K05163 
and JP23K03050.  
This work was also supported by the Research Institute for Mathematical Sciences, 
an International Joint Usage/Research Center located in Kyoto University. 
The author would like to thank Shota Inoue for his comments on the first draft.

%
%---------------------------------------

%---------------------------------------
%
%.......................................
\end{document}